\documentclass{article}

\usepackage{amsmath}
\usepackage{amssymb}
\usepackage{amsthm}
\newcommand{\e}{\varepsilon}
\newtheorem{theorem}{Theorem}[section]
\newtheorem{lemma}[theorem]{Lemma}
\newtheorem{remark}[theorem]{Remark}

\newtheorem{corollary}[theorem]{Corollary}
\numberwithin{equation}{section}
\usepackage{color}
\usepackage{comment}

\def\({\left(}
\def\){\right)}
\def\tcred{\relax}

\begin{document}
\title{A note on radial solutions to the critical Lane-Emden equation with a variable coefficient
}
\author{Daisuke Naimen${}^{*}$ and Futoshi Takahashi${}^{**}$}
\date{\small{\textit{Muroran Institute of Technology${}^{*}$, 
Osaka City University${}^{**}$
}}}
\maketitle
\begin{abstract}
In this note, we consider the following problem
\begin{equation*}
\begin{cases}
-\Delta u=(1+g(x))u^{\frac{N+2}{N-2}},\ u>0\text{ in }B,\\
u=0\text{ on }\partial B,
\end{cases}
\end{equation*}
where $N\ge3$ and $B\subset \mathbb{R}^N$ is a unit ball centered at the origin and $g(x)$ is a radial H\"{o}lder continuous function such that $g(0)=0$. 
We prove the existence and nonexistence of radial solutions by the variational method with the concentration compactness analysis and the Pohozaev identity. 
\end{abstract}

\section{Introduction}
We study the following problem
\begin{equation}
\begin{cases}
-\Delta u=(1+g(x))u^{\frac{N+2}{N-2}},\ u>0\text{ in }B\\
u=0\text{ on }\partial B,
\end{cases}\label{p}
\end{equation}
where $B\subset \mathbb{R}^N$ is a unit ball centered at the origin with $N\ge3$, $g$ is a locally H\"{o}lder continuous function in $\overline{B}$ and radial, i.e., $g(x)=g(|x|)$. 
We note that a typical case is given by $g(x)=|x|^\beta$ with $\beta\ge0$.  
We will show some existence and nonexistence results on \eqref{p}.

First let us consider the next basic problem which is extensively investigated by many authors;
\begin{equation}
\begin{cases}
-\Delta u=u^{\frac{N+2}{N-2}},\ u>0\text{ in }\Omega,\\
u=0\text{ on }\partial \Omega,
\end{cases}\label{p0}
\end{equation}
where $\Omega$ is a smooth bounded domain in $\mathbb{R}^N$ with $N\ge3$. 
Since the nonlinearity $u^{\frac{N+2}{N-2}}$ has the critical growth, as is well-known, due to the lack of the compactness of the associated Sobolev embedding $H_0^1(\Omega)\hookrightarrow L^{\frac{2N}{N-2}}(\Omega)$, 
the existence/nonexistence of solutions of \eqref{p0} becomes a very delicate and interesting question. 
In fact, in contrast to the subcritical case, we can prove that \eqref{p0} has no smooth solution if $\Omega$ is a star-shaped domain by the Pohozaev identity \cite{P} (See also \cite{BN}). 
Hence in order to ensure the existence of solutions of \eqref{p0}, we need some ``perturbation" to \eqref{p0}. 
A celebrated work in this direction is given by \cite{BN}. 
They add a lower \tcred{order} term $\lambda u^{q}$ $(1\le q<(N+2)/(N-2))$ to the critical nonlinearity $u^{\frac{N+2}{N-2}}$ (i.e., replace $u^{\frac{N+2}{N-2}}$ by $u^{\frac{N+2}{N-2}}+\lambda u^q$) 
and successfully show the existence of solutions of \eqref{p0}. 
After that, \cite{C}, \cite{BC} and \cite{Ba} prove that the topological perturbation to the domain can also induce solutions to \eqref{p0}. 
\tcred{See also \cite{Ding} \cite{Pas} for the effect of the geometric perturbation to the domain.}
Furthermore, another perturbation is found \tcred{by Ni} \cite{N}. 
He considers a variable coefficient $|x|^\alpha$ with $\alpha>0$ on $u^{\frac{N+2}{N-2}}$. 
More precisely he investigates
\begin{equation}
\begin{cases}
-\Delta u=|x|^\alpha u^{p},\ u>0\text{ in }B,\\
u=0\text{ on }\partial B,
\end{cases}\label{p1}
\end{equation}
where $\alpha>0$ and $p\in(1,\frac{N+2+2\alpha}{N-2})$. 
The crucial role of the variable coefficient $|x|^\alpha$ appears in the following compactness lemma for radially symmetric functions in $H^1_0(B)$. 
Here we define $H_r(B)$ is a subspace of $H_0^1(B)$ which consists of all radial functions.

\begin{lemma}[\cite{N}]\label{N} 
The map $u \tcred{\mapsto} |x|^m u$ from $H_r(B)$ to $L^p(B)$ is compact, for $p\in[1,\tilde{m})$ where
\[
\tilde{m}=\begin{cases} \frac{2N}{N-2-2m}\text{ if }m<\frac{N-2}{2}\\
             \infty\ \text{ otherwise }.
            \end{cases}
\]
\end{lemma} 

Applying this, one successfully obtains the existence of a mountain pass solution of \eqref{p1} for all  $p\in\left(1,\frac{N+2+2\alpha}{N-2}\right)$.
\tcred{The exponent} $p$ can be supercritical \tcred{(i.e., $p > \frac{N+2}{N-2}$)} if \tcred{$\alpha>0$}. 
We here note that, for the critical case, the essential point seems that $u^{\frac{N+2}{N-2}}$ has a variable coefficient which is radial and attains $0$ at the origin (see Example 2.1 in \cite{W}). 
In view of this it is an interesting question that \tcred{whether} it is possible to ensure the existence of solutions in the case where the coefficient does not attain $0$ at the origin. 
Very recently, Ai-Cowan \cite{AC} study another problem including our problem \eqref{p}. 
Applying their dynamical system approach, which is developed in \cite{A}, we can confirm the existence of radially symmetric solutions of \eqref{p} for the case $g(x)=|x|^\beta$ with $\beta\in(0,N-2)$. 
An interesting point in this case is that the coefficient $(1+g(x))$ attains the local minimum at the origin but not $0$. 
Hence we cannot apply Lemma \ref{N} directly. 
Then it is an interesting question to investigate how the coefficient can exclude the non-compactness of their nonlinearity. 
Motivated by this, we investigate \eqref{p} via the variational method. 
Our aim is to give a variational interpretation on the results in \cite{AC} and further, \tcred{to} extend their \tcred{results} to a more general coefficient which has a local minimum at the origin. 

Now in order to explain our main results, we give \tcred{an} observation to the results in \cite{AC}.
In the variational point of view, it seems better to write the right hand side of the equation of \eqref{p} as $u^{\frac{N+2}{N-2}}+g(x)u^{\frac{N+2}{N-2}}$. 
Then the first term is actually noncompact. 
On the other hand, the second one becomes compact by Lemma \ref{N} if $g(x)$ behaves like $|x|^\beta$ with $\beta>0$. 
Then we clearly expect that it would \tcred{play the role of} the subcritical perturbation $\lambda u^{q}$ with $1\le q<(N+2)/(N-2)$ in \cite{BN} mentioned above. 

Then, it is natural to consider the next more general problem. (See also the generalization in \cite{AC}.) 
\begin{equation}
\begin{cases}
-\Delta u=u^{\frac{N+2}{N-2}}+\lambda k(x) f(u),\ u>0&\text{ in }B,\\
u=0&\text{ on }\partial B\label{Q}
\end{cases}
\end{equation}
where $\lambda>0$ is a parameter 
and $k:\overline{B}\to \mathbb{R}$ and $f:\mathbb{R}\to \mathbb{R}$ satisfy some of  the next assumptions.
\begin{enumerate}
\item[(k1)] $k(x)\not\equiv 0$ is a nonnegative H\"{o}lder continuous function on $\overline{B}$ and radial, i.e., $k(x)=k(|x|)$.
\item[(k2)] $k(x)=O(|x|^{\beta})$ ($|x|\to 0$) for some $\beta>0$.
\item[(k3)] There exist constants $\gamma\ge\beta>0$ and $C,\delta>0$ such that $k(|x|)\ge C|x|^\gamma$ for all  $|x|\in(0,\delta)$.
\item[(f1)] $f(t)$ is locally H\"{o}lder continuous function on $[0,\infty]$ and $f(t)\ge0$ for all $t>0$ and $f(t)=0$ for all $t\le0$.
\item[(f2)] $\lim_{t\to 0}\frac{f(t)}{t}=0$ and $\lim_{t\to \infty}\frac{f(t)}{t^q}=0$ for  $q=(N+2+2\beta)/(N-2)$.
\item[(f3)] There exists a constant $\theta>2$ such that $f(t)t\ge \theta F(t)$ for all $t\ge0$ where $F(t):=\int_0^tf(s)ds$. 
\end{enumerate}

Now, we give our main results. 
\begin{theorem}\label{thm:c}We have the following.
\begin{enumerate}
\item[(i)] If $k,f$ satisfy (k1), (k2), (k3), (f1), (f2), (f3) and further, 
\begin{enumerate}
\item[(f4)] $\lim_{t\to \infty}\frac{f(t)}{t^p}=\infty$ for $p=\max\left\{1,\frac{2\gamma+6-N}{N-2}\right\}$,
\end{enumerate}
then \eqref{Q} admits a radially symmetric solution for all $\lambda>0$. 
\item[(ii)] If $k,f$ verify (k1), (k2), (f1), (f2), (f3) and further,
\begin{enumerate}
\item[(k4)] there exists a point $x_0\in \overline{B}$ such that $k(x_0)>0$ and,
\item[(f5)] there exists a constant $c>0$ such that $f(t)>0$ for all $t\in (0,c)$,
\end{enumerate}
then, there exists a constant $\lambda^*>0$ such that \eqref{Q} has a radially symmetric solution for all $\lambda>\lambda^*$.
\end{enumerate}
\end{theorem}

\begin{remark}
The hypothesis in \tcred{Theorem \ref{thm:c}} (i) permits the case where $k(x)=|x|^\beta$ for $\beta>0$ and $f(u)=u_+^{q}$ with any $q\in(\max\{1,(2\beta+6-N)/(N-2)\},(N+2+2\beta)/(N-2))$. 
The condition $q>\max\{1,(2\beta+6-N)/(N-2)\}$ is assumed to lower the mountain pass energy down to \tcred{the level for which the local compactness of the Palais-Smale sequences is valid.} 
See Lemmas \ref{lem:c30} and \ref{lem:c3} for the detail. 
On the other hand, (ii) is valid for  $f(u)=u_+^{q}$ with any $q\in(1,(N+2+2\beta)/(N-2))$.  
\end{remark}

\begin{remark} 
A similar problem is considered in \cite{CFM} and \cite{FGM}. 
The existence and nonexistence for the linear \tcred{perturbation case with} $k(r)=r^\beta$ \tcred{for} $\beta>0$ and $f(t)=t_+$ are completed by \cite{CFM}. 
Furthermore, the superlinear \tcred{perturbation case with} $k(r)=r^\beta$ \tcred{for} $\beta>0$ and $f(t)=t_+^{q}$ with $q\in(1,(N+2+2\beta)/(N-2))$ is treated in \cite{FGM}. 
Our theorem gives a generalization of a part of their \tcred{results}.
\end{remark}

A nonexistence result on \eqref{Q} is given by the Pohozaev identity as follows.
\begin{theorem}\label{thm:non0}
Let $\lambda\in\mathbb{R},$ $k(x)=|x|^\beta$ with $\beta\ge0$, $f(u)=u_+^q$ and $q\ge1$. 
Then \eqref{Q} admits no solution if one of the following is true;
\begin{enumerate}
\item[(i)] $q\in [1,(2\beta+N+2)/(N-2)]$ and $\lambda\le0$, or
\item[(ii)] $q\ge \frac{2\beta+N+2}{N-2}$ and $\lambda\ge0$, or,
\item[(iii)] $\beta=0$ and $q=(N+2)/(N-2)$.
\end{enumerate}
\end{theorem}

\begin{remark}
\tcred{The} same conclusion holds even if we replace the domain $B$ by any star-shaped domain. See the argument in Section \ref{sec:2}.
\end{remark}

Now we come back to our main question on \eqref{p}. 
The desired existence results are given as a corollary of (i) of Theorem \ref{thm:c}. 

\begin{corollary}\label{thm:0} We assume 
\begin{enumerate}
\item[(g1)] $g(x)$ is H\"{o}lder continuous and $g\ge-1$ on $\overline{B}$ and radial, i.e., $g(x)=g(|x|)$,
\item[(g2)] $g(0)=0$, and
\item[(g3)] there exist constants $\gamma\in(0,N-2)$, $\delta\in(0,1]$ and $C>0$ such that $g(|x|)\ge C|x|^\gamma$ for all $|x|\in(0,\delta)$.
\end{enumerate}
Then \eqref{p} admits at least one radially symmetric solution.
\end{corollary}

\begin{remark}
This theorem generalizes Theorem 2 \tcred{in \cite{AC}} for the case $g(|x|,u)=g(|x|)$.
To see this, note first that their condition (6) \tcred{in \cite{AC}} implies (g2) and (g3). 
Furthermore, since (g3) \tcred{is a condition} for the behavior of $g$ only near the origin, we \tcred{can} easily construct an example which \tcred{satisfies} (g2) and (g3)\tcred{,} but not (6). 
In addition, they prove \tcred{Theorem 2 in \cite{AC}} by dynamical system approach while we shall prove it via the variational method with the concentration compactness analysis. 
Hence our proof can give a variational interpretation and a generalization of \tcred{their theorem}. 
\end{remark}

By Corollary \ref{thm:0}, we have the existence of solution of \eqref{p} if $g(x)=\lambda|x|^\beta$ with $\beta\in(0,N-2)$ and $\lambda>0$. 
For the case including $\beta\ge N-2$, we have the next corollary \tcred{as} a direct consequence of (ii) in Theorem \ref{thm:c}. 

\begin{corollary}\label{thm:1}
Let $\lambda>0$, $g(x)=\lambda k(x)$ and $k(x)$ is a nonnegative H\"{o}lder continuous function in $\overline{B}$ such that $k(0)=0$ and $k(x)=k(|x|)$. Furthermore, assume there exists a point $x_0\in \overline{B}$ such that $k(x_0)>0$. Then there exists a constant $\lambda^*>0$ such that \eqref{p} admits at least one radially symmetric solution for all $\lambda>\lambda^*$. 
\end{corollary}

\begin{remark} 
\tcred{Corollary \ref{thm:1}} implies that if $g(x)=\lambda |x|^\beta$ with $\beta>0$, a radially symmetric solution exists for all sufficiently large $\lambda>0$. 
Furthermore, we remark that \tcred{this} generalizes Theorem 1 of \cite{AC}.
\end{remark}

The existence results above are \tcred{best possible} in the following sense. 
We have the following nonexistence result.
\begin{theorem}\label{thm:non1} 
Let $g(x)=\lambda|x|^\beta$ with $\beta\ge0$ and $\lambda\in\mathbb{R}$. 
Then \eqref{p} does not admit any radially symmetric solution if $\beta=0$ and $\lambda\in \mathbb{R}$\tcred{,} or $\beta\ge0$ and $\lambda\le0$. 
In addition if $\beta \ge N-2$, there exists a constant $\lambda_*>0$ which depends on $\beta$ and $N$ such that \eqref{p} has no radially symmetric solution for all $\lambda\in[0,\lambda_*]$. 
\end{theorem}

\begin{remark}
In our computation, we can choose 
\[
\lambda_*=\begin{cases}\frac{2(N-1)}{N-2}\text{ if }\beta=N-2,\\
                       \frac{2(N-1)}{N-2}\left(\frac{2N-2+\beta}{\beta-N+2}\right)^{\frac{\beta-N+2}{N-2}}\text{ if }\beta>N-2.
        \end{cases}
\]
For the detail, see the proof of Theorem \ref{thm:non1} in Section \ref{sec:2}. \end{remark}
\subsection*{Organization of this paper}
This paper consists of three sections with an appendix. In Section \ref{sec:1}, we give the proof of the existence results. 
In Section \ref{sec:2}, we show the nonexistence assertions by the Pohozaev identity. 
Lastly in Appendix \ref{ap}, we give a remark on the proof for the critical case for \tcred{the} reader's convenience. 
Throughout this paper we define $H_r(B)$ as a subspace of $H_0^1(B)$ which consists of all the radial functions. 
Furthermore we put $2^*=2N/(N-2)$ and define the Sobolev constant $S>0$ as usual by 
\tcred{
\[
S:=\inf_{u\in H^1_0(B)\setminus\{0\}}\frac{\| u \|^2}{\int_B|u|^{2^*}dx}
\]
where $\| u \|^2 = \int_B|\nabla u|^2dx$.}
Finally we define $B_s(0)$ as a $N$ dimensional ball centered at the origin with radius $s>0$.
\section{Existence results}\label{sec:1}

In this section, we give \tcred{a} proof of the existence results of our main theorems and corollaries. 
In the following we always suppose (k1), (k2), (f1) and (f2). 
\tcred{For the problem \eqref{Q}}, we define the associated energy functional
\[
I(u)=\frac{1}{2}\|u\|^2-\frac{1}{2^*}\int_Bu_+^{2^*}dx-\int_B kF(u)dx\ \ (u\in H_r(B)).
\]
Then noting our assumptions and Lemma \ref{N}, it is standard to see \tcred{that} $I(u)$ is well-defined and continuously differentiable on \tcred{$H_r(B)$}. 
 In addition, by (k1) and (f1),  the usual elliptic theory and the strong maximum principle ensure that every critical point of $I$ is a solution of \eqref{Q}. 
Hence our aim becomes to look for critical points of $I$. We first prove the mountain pass geometry of $I$ \cite{AR}.  

\begin{lemma}\label{lem:c1} 
We have
\begin{enumerate}
\item[(a)] $\exists \rho,a>0$ such that $I(u)\ge a$ for all $u\in H_r(B)$ with $\|u\|=\rho$, and
\item[(b)] for all $u\in H_r(B)\setminus\{0\}$, $I(tu)\to -\infty$ as $t\to \infty$,
\end{enumerate} 
for all $\lambda>0$. 
\end{lemma}

\begin{proof}
First note that by (f1) and (f2), we have that for any $\e>0$, there exists a constant $C>0$ such that $|f(t)|\le \e t+Ct^p$ for all $t\ge0$ and some $p\in(1,(N+2+2\beta)/(N-2))$.
Then Lemma \ref{N} and the Sobolev inequality \tcred{give} 
\[
I(u)\ge \left(\frac{1}{2}-\frac{\lambda \e}{\mu_1}\right)\|u\|^2-\lambda C\|u\|^{p+1}-C\|u\|^{2^*}
\]
for all $u\in H_r(B)$. 
Taking $\e\in(0,\mu_1/(4\lambda))$, we get (a) for all $\lambda\in(0,\infty)$. 

Next, since $k(x)f(u)\ge0$ for all $x\in \overline{B}$ and $u\in \mathbb{R}$, 
we have for all $t>0$ and $u\in H_r(B)\setminus\{0\}$ that
\[
I(tu)\le \frac{t^2}{2}\|u\|^2-\frac{t^{2^*}}{2^*}\int_Bu_+^{2^*}dx.
\]
Since $2<2^*$, we obtain $I(tu)\to -\infty$ as $t\to \infty$, which shows (b). 
This finishes the proof.
\end{proof}

Noting Lemma \ref{lem:c1}, we define  
\[
\Gamma:=\{\gamma\in C([0,1],H_r(B))\ |\ \gamma(0)=0,\ \gamma(1)=e\}
\]
with $e\in H_r(B)$ satisfying $\|e\|>\rho$ and $I(e)\le0$. Then we put 
\[
c_\lambda:=\inf_{\gamma\in \Gamma}\max_{u\in\gamma([0,1])}I(u).
\]
We next show the local compactness property of the Palais-Smale sequences of $I$. 
Here\tcred{, as usual,} we call $(u_n)\subset H_r(B)$ is a (PS)$_c$ sequence for $I$ 
if $I(u_n)\to c$ for some $c\in \mathbb{R}$ and $I'(u_n)\to0$ in $H_r^{-1}(B)$ as $n\to \infty$ where $H_r^{-1}(B)$ is \tcred{the} dual space of $H_r(B)$. 

\begin{lemma}\label{lem:c2} 
Suppose $f$ satisfies (f3) and $\lambda>0$. 
\tcred{If} $(u_n)\subset H_r(B)$ is a (PS)$_c$ sequence for a value $c<S^{N/2}/N$, \tcred{then} $(u_n)$ has a subsequence which strongly converges in $H_r(B)$ as $n\to \infty$. 
\end{lemma}

\begin{proof}
By (f3), we obtain that 
\[
\begin{split}
c+o(1)&=I(u_n)-\frac{1}{\min\{2^*,\theta\}}\langle I'(u_n),u_n\rangle+o(1)\|u_n\|\\
        &\ge \left(\frac{1}{2}-\frac{1}{\min\{2^*,\theta\}}\right)\|u_n\|^2+o(1)\|u_n\|
\end{split}
\]
This shows the claim. Hence noting (f1), (f2) and Lemma \ref{N}, we have that, up to a subsequence, there exists a nonnegative function $u\in H_r(B)$ such that
\begin{equation}
\begin{cases}
u_n\rightharpoonup u\text{ weakly in }H_0^1(B),\\
\int_B kf(u_n)dx\to \int_Bkf(u)dx,\\
\int_B r^\beta |u_n|^{s+1}dx\to \int_Br^\beta |u|^{s+1}dx\text{ for any }s\in[1,(N+2+2\beta)/(N-2)),\\
u_n\to u\text{ a.e. on }B,
\end{cases} \label{eq:d}
\end{equation}
as $n\to \infty$. 
Furthermore, since $(u_n)\subset H_r(B)$, the concentration compactness lemma (Lemma I.1 in \cite{L}) implies that there exist values \tcred{$\nu_0$}, $\mu_0\ge0$ such that
\[
\begin{split}
&|\nabla u_n|^2 \rightharpoonup d\mu \ge |\nabla u|^2+\mu_0\delta_0,\\
&(u_n)_+^{2^*}\rightharpoonup d\nu =u^{2^*}+\nu_0 \delta_{0},
\end{split} 
\]
in the measure sense where $\delta_0$ denotes the Dirac measure with mass $1$ which concentrates at $0\in \mathbb{R}^N$ and 
\begin{equation}
S\nu_0^{\frac{2}{2^*}}\le \mu_0.\label{eq:d1}
\end{equation}
Let us show $\nu_0=0$. If not, we define a smooth test function $\phi$ in $\mathbb{R}^N$ such that $\phi=1$ on $B(0,\varepsilon)$, $\phi=0$ on $B(0,2\varepsilon)^c$ and $0\le \phi\le1$ otherwise. We also assume $|\nabla \phi|\le 2/\varepsilon$. Then noting (f1), (f2) and using (k1), (k2), \eqref{eq:d} and Lemma \ref{N}, we get 
\[
\begin{split}
0&=\lim_{n\to \infty}\langle I'(u_n),u_n\phi\rangle\\
&=\lim_{n\to \infty}\left(\int_B\nabla u_n \tcred{\cdot} \nabla (u_n\phi)dx-\int_B (u_n)_+^{2^*}\phi dx-\lambda\int_B kf(u_n)u_n\phi dx\right)\\
&=\lim_{n\to \infty}\left(\int_B|\nabla u_n|^2\phi dx-\int_B (u_n)_+^{2^*}\phi dx-\lambda\int_B kf(u_n)u_n\phi dx+\int_B u_n\nabla u_n \tcred{\cdot} \nabla \phi dx\right)\\
&= \int_{\overline{B}}\phi d\mu-\int_{\overline{B}}\phi d\nu+o(1)
\end{split}
\]
where $o(1)\to 0$ as $\varepsilon\to 0$. It follows that
\[
0\ge \mu_0-\nu_0.
\]
Then by \eqref{eq:d1}, we obtain
\[
\nu_0\ge S^{\frac{N}{2}}.
\]
Using this estimate, we have by (f3) that
\[
\begin{split}
c&=\lim_{n\to \infty}\left(I(u_n)-\frac12\langle I'(u_n),u_n\rangle\right)\\
 &\ge\frac{1}{N}\lim_{n\to \infty}\int_{\overline{B}}d\nu\\
  &\ge \frac{S^{\frac{N}{2}}}{N}
\end{split}
\]
which contradicts our assumption. It follows that
\[
\lim_{n\to \infty}\int_B (u_n)_+^{2^*}dx= \int_B u^{2^*}dx.
\]
Then the usual argument proves $u_n\to u$ in $H_r(B)$. We finish the proof.
\end{proof}

Next we estimate the mountain pass energy $c_\lambda$. 
To do this, we use the Talenti function $U_\e(x):=\frac{\e^{\frac{N-2}{2}}}{(\e^2+|x|^2)^{\frac{N-2}{2}}}$ \cite{Ta}. 
Moreover we define a cut off function $\psi \in C^{\infty}_0(B)$ such that $\psi(x)=\psi(|x|)$, $\text{supp}\{\psi\}\subset B_{\delta}(0)$ and $\psi=1$ on $B_{\eta}(0)$ for some  $\eta\in(0,\delta)$. 
We set $u_\e:=\psi U_\e$ and $v_\e:=u_\e/\|u_\e\|_{L^{2^*}(B)}\in H_r(B)$. 
Then, if $q>\max(2\gamma+6-N)/(N-2)$, a similar calculation with that in \cite{BN} \tcred{shows} that
\begin{equation}
\begin{cases}
\|v_\e\|^2=S+O(\e^{N-2})\\
\|v_\e\|_{L^{2^*}(B)}=1,\\
\int_Bkv_\e^{q+1}dx\ge C\int_B|x|^\gamma v_\e^{q+1}dx= C' \e^{a}+O(\e^{N-2})
\end{cases}\label{v}
\end{equation}
where $a=\gamma+N-\frac{(N-2)(q+1)}{2}$ and $C,C'>0$ are constants. 
Let us prove the next lemma (Cf. Lemma 2.1 in \cite{BN}.)

\begin{lemma}\label{lem:c30}
Assume that $k$ verifies (k3). Then if  
\begin{equation}
\lim_{\e\to0}\e^{\gamma+2}\int_0^{\e^{-1}} F\left[\left(\frac{\e^{-1}}{1+r^2}\right)^{\frac{N-2}{2}}\right]r^{\gamma+N-1}dr=\infty\label{eq:c30}
\end{equation}
holds, we have $c_\lambda<\tcred{S^{N/2}/N}$ for all $\lambda>0$.
\end{lemma}

\begin{proof} Let $v_\e\in H_r(B)$ as above. 
Then from Lemma \ref{lem:c1}, we find a constant $t_\e>0$ such that $I(t_\e v_\e)=\max_{t\ge0}I(tv_\e)$. 
Since 
\[
	\tcred{0=\frac{d}{dt}|_{t=t_\e}I(tv_\e)= t_\e \|v_\e\|^2 - t_\e^{2^*-1} - \int_B kf(t_\e v_\e) v_\e dx}
\]
and \tcred{$\int_B k f(v_\e) v_\e dx\ge0$} by (k1) and (f1), we have 
\[
t_\e\le \|v_\e\|^{\frac{2}{2^*-2}}=:T_\e. 
\]
Since $\tcred{T_\e}=\|v_\e\|^{\frac{2}{2^*-2}}$ is the maximum point of the map $t\mapsto \frac{t^2}{2}\|v_\e\|^2-\frac{t^{2^*}}{2^*}$, 
we get by \eqref{v} that for any $t>0$
\[
\begin{split}
I(t v_\e)&\le I(t_\e v_\e)\\
           &\le \frac{T_\e^2}{2}\|v_\e\|^2-\frac{T_\e^{2^*}}{2^*}-\int_B kF(t_\e v_\e)dx\\
          &\le  \frac{S^{\frac{N}{2}}}{N}-\int_B kF(t_\e v_\e)dx+O(\e^{N-2}).
\end{split}
\]
Therefore once we prove
\begin{equation}
\lim_{\e\to0}\e^{-(N-2)}\int_BkF(t_\e v_\e)dx=\infty,\label{eq:c00}
\end{equation}
we conclude $c_\lambda\le I(t_\e v_\e)<S^{N/2}/N$ for all small $\e>0$. This completes the proof. 
Lastly let us ensure \eqref{eq:c00}. 
To do this, we first claim that \tcred{$\lim_{\e\to0}t_\e \to S^{(N-2)/4}$}.
Indeed, using (f2), for any $\delta>0$, there exists a constant $C_\delta>0$ such that  
\[
\int_B \frac{kf(t_\e v_\e)v_\e}{t_\e}dx\le t_\e^{q-1}\delta\int_B |x|^\beta v_\e^{q+1}dx+C_\delta \int_B|x|^\beta v_\e^2dx. 
\]
Since $t_\e\le T_\e=O(1)$, $\int_B |x|^\beta v_\e^{\tcred{q+1}}dx=O(1)$ by $q=(N+2+2\beta)/(N-2)$ and $\int_B|x|^\beta v_\e^2dx=o(1)$ as $\e\to0$, we get 
\[
\lim_{\e\to0}\int_B \frac{kf(t_\e v_\e)v_\e}{t_\e}dx=0.
\]
Then since $\langle I'(t_\e v_\e),v_\e \rangle=0$, \tcred{we have}
\[
t_\e=\left(\|v_\e\|^2-\int_B\frac{kf(t_\e v_\e)v_\e}{t_\e} dx\right)^{\tcred{\frac{1}{2^*-2}}}.
\]
\tcred{This with \eqref{v} proves} the claim. 
In particular, $t_\e$ converges to a positive value as $\e\to0$. Now we calculate by (k3) that 
\[
\begin{split}
\e^{-(N-2)}\int_BkF(t_\e v_\e)dx&\ge C_1\e^{-(N-2)}\int_{0}^\eta  F\left[t_\e\left(\frac{\e}{\e^2+r^2}\right)^{\frac{N-2}{2}}\right]r^{\gamma+N-1}dr\\
&\ge C_2\e^{\gamma+2}\int_0^{\frac{\eta}{\e}} F\left[t_\e\left(\frac{\e^{-1}}{1+r^2}\right)^{\frac{N-2}{2}}\right]r^{\gamma+N-1}dr\\
&\ge C_3\e^{\gamma+2}\int_0^{\frac{D}{\e}}F\left[\left(\frac{\e^{-1}}{1+r^2}\right)^{\frac{N-2}{2}}\right]r^{\gamma+N-1}dr
\end{split}
\]
for some constant $C_1,C_2,C_3,D>0$ where in the last inequality we replace $\e/t_\e^{(N-2)/2}$ by $\e$ which does not change the conclusion below. 
If $D\ge1$, we clearly get \eqref{eq:c00} by our assumption \eqref{eq:c30}. 
If $D<1$, we obtain
\[
\begin{split}
\e^{\gamma+2}\int_0^{\frac{D}{\e}}F\left[\left(\frac{\e^{-1}}{1+r^2}\right)^{\frac{N-2}{2}}\right]r^{\gamma+N-1}dr=
&\e^{\gamma+2}\int_0^{\frac{1}{\e}}F\left[\left(\frac{\e^{-1}}{1+r^2}\right)^{\frac{N-2}{2}}\right]r^{\gamma+N-1}dr\\
&-\e^{\gamma+2}\int_{\frac{D}{\e}}^{\frac{1}{\e}}F\left[\left(\frac{\e^{-1}}{1+r^2}\right)^{\frac{N-2}{2}}\right]r^{\gamma+N-1}dr.
\end{split}
\]
Finally, note that (f2) shows 
\[
\e^{\gamma+2}\int_{\frac{\tcred{D}}{\e}}^{\frac{1}{\e}}F\left[\left(\frac{\e^{-1}}{1+r^2}\right)^{\frac{N-2}{2}}\right]r^{\gamma+N-1}dr=o(1)
\]
where $o(1)\to0$ as $\e\to0$. This finishes the proof.
\end{proof}

The next lemma \tcred{confirms} that under our \tcred{assumptions}, $f(t)$ satisfies \eqref{eq:c30}.

\begin{lemma}\label{lem:c3} 
Assume (k3). Then, if $f$ satisfies (f4), then \eqref{eq:c30} holds true.
\end{lemma}

\begin{proof} By (f4), for any $M>0$, there exists a constant $R>0$ such that $f(t)\ge M t^{p}$ where $p=\max\{1,\frac{2\gamma+6-N}{N-2}\}$. 
Furthermore, note that if $r\le C\e^{-1/2}$ for $C=(2R)^{-(N-2)/2}$, we get
\[
\left(\frac{\e^{-1}}{1+r^2}\right)^{\frac{N-2}{2}}\ge R
\]
for all small $\e>0$. It follows that
\[
\begin{split}
\e^{\gamma+2}\int_0^{\e^{-1}}F\left[\left(\frac{\e^{-1}}{1+r^2}\right)^{\frac{N-2}{2}}\right] &r^{\gamma +N-1}dr 
\ge \e^{\gamma+2}\int_0^{C\e^{-\frac12}}F\left[\left(\frac{\e^{-1}}{1+r^2}\right)^{\frac{N-2}{2}}\right]r^{\gamma +N-1}dr\\
&\ge \e^{\gamma+2} \frac{M}{p+1} \int_0^{C\e^{-\frac12}}\left(\frac{\e^{-1}}{1+r^2}\right)^{\frac{(N-2)(p+1)}{2}}r^{\gamma +N-1}dr\\
&\to \infty
\end{split}
\]
as $\e\to0$. This completes the proof.
\end{proof}

\begin{lemma}\label{lem:c4} 
If $k,f$ satisfy (k4) and (f5), we have a constant $\lambda^*>0$ such that $c_\lambda<S^{N/2}/N$ for all $\lambda>\lambda^*$.
\end{lemma}

\begin{proof} Since $k(x_0)>0$ by (k4), there exist constants $0<r_1<|x_0|<r_2<1$ such that $k>0$ on $\overline{B(0,r_2)}\setminus B(0,r_1)$. Then we choose a radial function $u\in C_0^{\infty}(B)\setminus\{0\}$ such that $u\ge0$ and $\text{supp}\{u\}\subset \overline{B(0,r_2)}\setminus B(0,r_1)$. Then by Lemma \ref{lem:c1}, we have a constant $t_\lambda>0$ such that $I(t_\lambda u)=\max_{t>0}I(tu)$. Since $\frac{d}{dt}|_{t=t_\lambda}I(tu)=0$, we get 
\[
\|u\|^2-t_\lambda^{2^*-2}\int_Bu_+^{2^*}dx-\lambda\int_B\frac{kf(t_\lambda u)u}{t_\lambda}dx=0
\]
It follows that $t_\lambda\to0$ as $\lambda\to \infty$. If not, there exists a sequence $(\lambda_n)\subset (0,\infty)$ such that $\lambda_n\to \infty$ and $t_{\lambda_n}\to t_0>0$ for some value $t_0>0$ as $n\to \infty$. But this is impossible in view of the previous formula and (f5). Then it follows from (k1) and (f1) that
\[
\begin{split}
c_\lambda\le I(t_\lambda u)\le t_\lambda^2\|u\|^2\to 0
\end{split}
\]
as $\lambda\to \infty$. This finishes the proof.
\end{proof}

Then we prove the existence assertions of main theorems.

\begin{proof}[Proof of Theorem \ref{thm:c}] 
First note that under the assumption in Lemma \ref{lem:c1} \tcred{and} the mountain pass theorem (\cite{AR}, see also Theorem 2.2 in \cite{BN}), 
there exists a (PS)$_{c_\lambda}$ sequence $(u_n)\subset H_r(B)$ of $I$. Hence our aim is to see that $(u_n)$ has a subsequence which strongly converges in $H_r(B)$. 
\tcred{This fact} follows from Lemmas \ref{lem:c1}, \ref{lem:c2}\tcred{,} \ref{lem:c30} and \ref{lem:c3}\tcred{, which proves (i).}
The proof of (ii) is completed by Lemmas \ref{lem:c1}, \ref{lem:c2} and \ref{lem:c4}.  This completes the proof \tcred{of Theorem \ref{thm:c}.} 
\end{proof}

\begin{proof}[Proof of Corollary \ref{thm:0}] 
The proof is clear from (i) of Theorem \ref{thm:c}. Here we remark on (g1) and (g2). 
We first note that \tcred{non-negativity} of $k$ in (k1) is \tcred{needed} only to apply the maximum principle. 
Hence it is clear that in the present case it can be weakened to $g\ge-1$ in (g1). 
Furthermore, by (g1), the associated energy functional
\[
I(u)=\frac12 \|u\|^2-\frac{1}{2^*}\int_B(1+g)|u|^{2^*}dx
\]  
is always well-defined. Hence we can weaken (k2) in Theorem \ref{thm:c} to the condition $k(0)=0$. 
Finally, in the present case, since we do not assume $k(|x|)=O(|x|^\beta)$ for $\beta>0$, in principle, we cannot use Lemma \ref{N} directly in the proof of Lemma \ref{lem:c2}. 
Although the modification is trivial, we will give the modified proof in Appendix \ref{ap} for \tcred{the} readers' convenience. 
\end{proof}

\begin{proof}[Proof of Corollary \ref{thm:1}] 
The proof is immediate by (ii) of Theorem \ref{thm:c}. 
\end{proof}
\section{Nonexistence results}\label{sec:2}

In this section, we prove the nonexistence results by the Pohozaev identity. 
Since some results still hold true for the star-shaped domain, we first consider the problem
\begin{equation}
\begin{cases}
-\Delta u=|u|^{2^*-2}u+g|u|^{q-1}u\text{ in }\Omega\\
u=0\text{ on }\partial \Omega,
\end{cases}\label{eq:non}
\end{equation}
where $\Omega\subset \mathbb{R}^N$ with $N\ge3$ is a bounded smooth domain, $q\ge1$ and $g$ is \tcred{a} $C^1$ function. 
\tcred{Now, let us recall the formula}
\begin{equation}
\int_\Omega\left\{\frac{x\cdot \nabla g}{q+1}+\left(\frac{N}{q+1}-\frac{N-2}{2}\right)g\right\}|u|^{q+1}dx=\frac12\int_{\partial \Omega}(x\cdot \nu)|\nabla u|^2ds_x
\label{eq:pov}
\end{equation}
\tcred{holds} for any solution $u\in C^1(\overline{\Omega})$. 
\tcred{This is the Pohozaev identity for \eqref{eq:non}.}

\begin{theorem}
Let $\lambda\in \mathbb{R}$ and  $g(x)=\lambda |x|^\beta$ with $\beta\ge0$. Then if $\Omega$ is a star-shaped domain, \eqref{eq:non} has no  $C^1$ solution if either one of the following holds;
\begin{enumerate}\label{thm:non2}
\item[(i)] $\lambda\le0$ and $q\le (N+2+2\beta)/(N-2)$ or,
\item[(ii)]  $\lambda\ge0$ and  $q\ge (N+2+2\beta)/(N-2)$ or otherwise,
\item[(iii)] $\beta=0$, $\lambda\in\mathbb{R}$ and $q=(N+2)/(N-2)$.
\end{enumerate}
\end{theorem}
\begin{proof} Let $u\in C^1(\overline{\Omega})$ be a solution of \eqref{eq:non}. Then under the assumption in the theorem, we get by \eqref{eq:pov} that
\[
\lambda\int_\Omega\left(\frac{\beta+N}{q+1}-\frac{N-2}{2}\right)|x|^\beta |u|^{q+1}dx=\frac12\int_{\partial \Omega}(x\cdot \nu)|\nabla u|^2ds_x.
\]
Then if one of (i)-(iii) holds, the left hand side is nonpositive. On the other hand, since $x\cdot \nu\ge0$ by our assumption, we have $|\nabla u|\equiv 0$ on $\partial \Omega$. Then from the principle of unique continuation we must have $u\equiv 0$ in $\Omega$. This shows the proof. 
\end{proof}

\begin{proof}[Proof of Theorem \ref{thm:non0}] 
The proof is a direct consequence of Theorem \ref{thm:non2}.
\end{proof}

Lastly let us show the proof of Theorem \ref{thm:non1}. 
To do this, we assume $q\ge1$ and  $u=u(r)$ $(r\in[0,1])$ is a solution of 
\begin{equation}
\begin{cases}
-u''-\displaystyle\frac{(N-1)}{r}u'=|u|^{\frac{4}{N-2}}u+g|u|^{q-1}u\text{ in }(0,1),\\
u'(0)=0=u(1).
\end{cases}\label{rad2}
\end{equation}
with a $C^1$ function $g(r)$ on $[0,1]$. In addition, we suppose $\psi(r)$ $(r\in[0,1])$ is a smooth test function such that $\psi(0)=0$. Then we have the following. (See \cite{BN} and also \cite{GG}.) 
\begin{theorem}\label{thm:p} If $u$ is a solution of \eqref{rad2}, we get
\begin{equation}
\begin{split}
\psi(1)|u'(1)|^2&=\frac12\int_0^1 u^2 r^{N-4}\left\{r^3\psi'''-(N-1)(N-3)r \psi'+(N-1)(N-3)\psi\right\}dr\\
&+\frac{2(N-1)}N\int_0^1|u|^{2^*}(r^{N-1}\psi'-r^{N-2}\psi)dr\\
&\hspace{-5em}+\frac{1}{q+1}\int_0^1|u|^{q+1}\left\{(q+3)gr^{N-1}\psi'-(q-1)(N-1)gr^{N-2}\psi+2g' r^{N-1}\psi\right\}dr.
\label{f3}
\end{split}
\end{equation}
\end{theorem}
\begin{proof}
Multiplying the equation in \eqref{rad2} by $r^{N-1}\psi u'$ gives
\begin{equation}
\begin{split}
&\psi(1)|u'(1)|^2-\int_0^1|u'|^2\left\{r^{N-1}\psi'-(N-1)r^{N-2}\psi\right\}dr\\
&=\frac{N-2}{N}\int_0^1|u|^{2^*}\left\{r^{N-1}\psi'+(N-1)r^{N-2}\psi\right\}dr\\
&\ \ \ +\frac{\lambda(q+1)}{2}\int_0^1|u|^{q+1}\left\{g' r^{N-1}\psi+r^{N-1}g\psi'+(N-1)r^{N-2}g\psi\right\}dr\tcred{.}
\label{f1}
\end{split}
\end{equation}
On the other hand, we multiply the equation in \eqref{rad2} by $(r^{N-1}\psi'-(N-1)r^{N-2}\psi )u$ and compute
\begin{equation}
\begin{split}
&\int_0^1|u'|^2\left\{r^{N-1}\psi'-(N-1)r^{N-2}\psi\right\}dr\\
&-\frac12\int_0^1u^2\left\{r^{N-1}\psi'''+(N-1)(N-3)r^{N-4}(\psi-r\psi')\right\}dr\\&=\int_0^1 |u|^{2^*}\left\{r^{N-1}\psi'-(N-1)r^{N-2}\psi\right\}dr\\
&\ \ \ +\lambda\int_0^1 g(r)|u|^{q+1}\left\{r^{N-1}\psi'-(N-1)r^{N-2}\psi\right\}dr.
\end{split}
\label{f2}\end{equation}
Combining \eqref{f1} and \eqref{f2}, we complete the proof.
\end{proof}

\begin{proof}[Proof of \tcred{Theorem} \ref{thm:non1}] 
The first assertion follows from Theorem \ref{thm:non2}.  
Let us prove the second assertion. 
To do this, assume $\lambda>0$ and \tcred{$u = u(r)$ $(r \in [0,1])$} is a radially symmetric solution of \eqref{p}. 
Then it satisfies
\begin{equation}
\begin{cases}
-u''-\displaystyle\frac{(N-1)}{r}u'=(1+g)|u|^{\frac{4}{N-2}}u\text{ in }(0,1),\\
u'(0)=u(1)=0,
\end{cases}\label{rad}
\end{equation}
where we put $g(r)=\lambda r^{\beta}$. 
Again choose a smooth test function $\psi$ such that $\psi(0)=0$. Then by Theorem \ref{thm:p}, we have 
\begin{equation}
\begin{split}
&\frac12\int_0^1 u^2 r^{N-4}\left\{r^3\psi'''-(N-1)(N-3)r \psi'+(N-1)(N-3)\psi\right\}dr\\
&=\psi(1)|u'(1)|^2 \\
&+\frac1N\int_0^1|u|^{2^*}\left\{-(N-2)g' r^{N-1}\psi+2(N-1)(1+g(r))(r^{N-2}\psi-r^{N-1}\psi')\right\}dr.
\label{e3}
\end{split}
\end{equation}
\tcred{We} fix $\beta\ge N-2$ and then select $\psi(r)=ar^{N-1}+br$ so that $r^3\psi'''-(N-1)(N-3)r \psi'+(N-1)(N-3)\psi=0$ and $\psi(0)=0$. 
This ODE has an explicit solution $\psi(r)=ar^{N-1}+br+cr^{-(N-3)}$ where $a,b,c\in \mathbb{R}$ are arbitrary constants. 
Since we assume $\psi(0)=0$, we must have $c=0$, i.e., $\psi(r)=ar^{N-1}+br$. 
Then we get
\begin{equation}
\begin{split}
&\psi(1)|u'(1)|^2 \\
&+\frac1N\int_0^1|u|^{2^*}\left\{-(N-2)g' r^{N-1}\psi+2(N-1)(1+g(r))(r^{N-2}\psi-r^{N-1}\psi')\right\}dr=0.
\label{e4}
\end{split}
\end{equation} 
Substituting $\psi(r)=ar^{N-1}+br$ into 
\[
	h(r):=-(N-2)k' r^{N-1}\psi+2(N-1)(1+k)(r^{N-2}\psi-r^{N-1}\psi'), 
\]
we \tcred{see}
\begin{align*}
&h(r)=r^{2N-3} \times \\
&\times \left[-\lambda a (N-2)\left\{2(N-1)+\beta\right\}r^{\beta}-\lambda b\beta (N-2) r^{\beta-N+2}-2a(N-1)(N-2)\right].
\end{align*}
Finally, we choose $a<0$ and $b=|a|>0$. In particular, we have $\psi(1)=a+b=0$. 
Then some elementary calculations show that if we set
\[
\lambda_*=\begin{cases}\frac{2(N-1)}{N-2}\text{ if }\beta=N-2,\\
                       \frac{2(N-1)}{N-2}\left(\frac{2N-2+\beta}{\beta-N+2}\right)^{\frac{\beta-N+2}{N-2}}\text{ if }\beta>N-2,
        \end{cases}
\]
we \tcred{assure that} $h\not=0$ and $h\ge0$ for all $\lambda\in[0,\lambda_*]$. 
Therefore in view of \eqref{e4}, we \tcred{reach to} a contradiction if $\lambda\in[0,\lambda_*]$. 
This finishes the proof.
\end{proof}
\appendix

\section{Critical case}\label{ap}

In this appendix, we give a proof of Lemma \ref{lem:c2} under the assumption in Corollary \ref{thm:0} for \tcred{the} readers' convenience. 
Especially we will use only the condition (g2) which is weaker than (k2).
\begin{lemma}\label{compact}
Assume (g1), (g2) and $(u_n)\subset H_r(B)$ is a (PS)$_c$ sequence of
\[
I(u)=\frac{1}{2}\|u\|^2-\frac{1}{2^*}\int_B(1+g)u_+^{2^*}dx\tcred{.}
\]
Then if $c<S^{\frac{N}{2}}/N$, $(u_n)$ has a subsequence which strongly converges in $H_r(B)$. 
\end{lemma}
\begin{proof} From the definition we have
\[
\begin{split}
c+o(1)&=I(u_n)-\frac{1}{2^*}\langle I'(u_n),u_n\rangle+o(1)\|u_n\|\\
        &\ge \frac1N \|u_n\|^2+o(1)\|u_n\|.
\end{split}
\]
This implies $(u_n)$ is bounded in $H_r(B)$. Then  we can assume tcred{that} there exists a nonnegative function $u\in H_r(B)$ such that
\[
\begin{cases}
u_n\rightharpoonup u\text{ weakly in }H_r(B),\\
u_n\to u\text{ a.e. on }B,
\end{cases}
\]
up to a subsequence. 
\tcred{By} the concentration compactness lemma, we can suppose that there exist values $\mu_0,\nu_0\ge0$ such that
\[
\begin{cases}
|\nabla u_n|^2 \rightharpoonup d\mu \ge |\nabla u|^2+\mu_0\delta_{k},\\
u_n\to u\text{ in }L^{p}(B)\text{ for all }p\in(1,2N/(N-2)),\\
(u_n)_+^{2^*}\rightharpoonup d\nu =u_+^{2^*}+\nu_0 \delta_{0},
\end{cases} 
\]
in the measure sense, where $\delta_{0}$ denotes the Dirac delta measure concentrated at the origin with mass $1$ as before.  
Furthermore, we have
\begin{equation}
S\nu_0^{\frac{2}{2^*}}\le\mu_0.\label{c1}
\end{equation}
We show $\nu_0=0$. 
To this end, we assume $\nu_0>0$ on the contrary. 
Then, for small $\e>0$, we define a smooth test function $\phi$ as in the proof of Lemma \ref{lem:c2}. 
Since $I'(u_n)\to 0$ in $H^{-1}(B)$ and $(u_n)$ is bounded, we have 
\[
\begin{split}
0&=\lim_{n\to \infty}\langle I'(u_n),u_n\phi\rangle\\
&=\lim_{n\to \infty}\left(\int_B\nabla u_n \tcred{\cdot} \nabla (u_n\phi)dx-\int_B (1+g)(u_n)_+^{2^*}\phi dx\right)\\
&=\lim_{n\to \infty}\left(\int_B|\nabla u_n|^2\phi dx-\int_B (1+g)(u_n)_+^{2^*}\phi dx+\int_B u_n\nabla u_n \tcred{\cdot} \nabla \phi dx\right)\\
&= \int_{\overline{B}}\phi d\mu-\int_{\overline{B}}(1+g)\phi d\nu+o(1)
\end{split}
\]
where $o(1)\to 0$ as $\varepsilon\to 0$. Taking $\varepsilon\to 0$ and noting $g(0)=0$, we obtain
\[
0\ge\mu_0-\nu_0.
\]
Then using \eqref{c1}, we get 
\[
\nu_0\ge S^{\frac N2}.\label{c2}
\]
Finally, noting this estimate, we see
\[
\begin{split}
c&=\lim_{n\to \infty}\left(I(u_n)-\frac12\langle I'(u_n),u_n\rangle\right)\\
 &=\frac{1}{N}\lim_{n\to \infty}\int_{\overline{B}}(1+g)d\nu\\
  &\ge \frac{S^{\frac{N}{2}}}{N}
\end{split}
\]
since $g(0)=0$, which is a contradiction. It follows that
\[
\lim_{n\to \infty}\int_B (1+g)(u_n)_+^{2^*}dx=\int_B (1+g)u_+^{2^*}dx.
\]
Then a standard argument shows that $u_n\to u$ in $H_r(B)$. This completes the proof.
\end{proof}

\subsection*{Acknowledgment} 
This work is inspired by the talk by Prof. Ai and Prof. Cowan at AMS sectional meeting at \tcred{Vanderbilt University} in April, 2018. 
The authors thank to them for their favorable discussion on this note. 
\tcred{This work is partly supported by Osaka City University Advanced Mathematical Institute (MEXT Joint Usage/Research Center on Mathematics and Theoretical Physics).
The first author (D.N.) is also supported by JSPS Kakenhi 17K14214 Grant-in-Aid for Young Scientists (B). The second author (F.T.) is also supported by JSPS Kakenhi 19136384 Grant-in-Aid for Scientific Research (B).}

\end{document}